\documentclass[11pt]{article}
\usepackage{amssymb, amsmath, hyperref, array}

\def\C{\mathbb{C}}

\def\N{\mathbb{N}}

\newtheorem{defn}{Definition}

\newtheorem{lemma}[defn]{Lemma}

\newtheorem{theorem}[defn]{Theorem}

\newtheorem{remark}[defn]{Remark}

           {\vspace{3.3mm}
           \noindent{\bf #1}\it}%
           {\vspace{3.3mm}}

\newenvironment{proof}[1]{
  \trivlist \item[\hskip \labelsep{\it #1}]}{\hfill\mbox{$\square$}
  \endtrivlist}

\def\sD { {\rm  sDisc}}

\title{On the Davenport-Mahler bound}

\author{Paula Escorcielo\footnote
{Partially supported by the Argentinian grant UBACYT 20020120100133.} 
\qquad
Daniel Perrucci$^{*}$\footnote
{Partially supported by the Argentinian  grant 
PIP 2014-2016 11220130100527CO CONICET.} 
\\[3mm]
{\small Departamento de Matem\'atica, FCEN, Universidad de Buenos Aires, Argentina}\\ 
{\small  IMAS, CONICET--UBA, Argentina}
}

\begin{document}

\maketitle

\begin{abstract}
We prove that the Davenport-Mahler bound holds for arbitrary graphs with vertices on the set of roots of a given 
univariate polynomial with complex coefficients.  
\end{abstract}

\section*{Introduction}

The Davenport-Mahler bound is a 
lower bound for the product of the lengths of the edges on a graph whose vertices
are the complex roots of a given univariate polynomial $P \in \C[X]$, under certain assumptions. 
Its origins are the work of Mahler (\cite{Mahler}), where a lower bound for the minimum separation between 
two roots of $P$ in terms of the discriminant of $P$ is given, and the work of Davenport (see 
\cite[Proposition 8]{Davenport}), where for the first time
a lower bound for the joint product of many different distances between roots of $P$ (which is not simply the product
of a lower bound for each distance) is obtained. Roughly speaking, this bound makes evident an interaction between the 
involved distances, in the sense that if some of them are very small, the rest cannot be that small.

Throughout the literature, there are different versions of this bound. 
We include here the one from \cite[Theorem 3.1]{ESY} (see also  \cite{John, Yap}). We refer the reader to 
\cite{MigSte} for the definition of discriminant and Mahler measure.

\begin{theorem}[Davenport-Mahler bound]\label{th:DM}
Let $P \in \C[X]$ be a polynomial of degree $d$.
Let $G=(V,E)$ be a directed graph whose vertices $\{v_1, \dots, v_k\}$ are a subset of the roots of $P$ 
such that: 
\begin{enumerate}
 \item if $(v_i, v_j) \in E$, then $|v_i| \le |v_j|$,
 \item $G$ is acyclic,
 \item the in-degree of any vertex is at most 1. 
\end{enumerate}
Then
$$
\prod_{(v_i, v_j) \in E} |v_i - v_j| \ \ge \ |{\rm Disc}(P)|^{1/2} \, {\rm M}(P)^{-(d-1)} \, \Big(\frac{d}{\sqrt{3}}\Big)^{-\#E} \,d^{-d/2},
$$
where 
${\rm Disc}(P)$ and ${\rm M}(P)$ are the discriminant and the Mahler measure of $P$. 
\end{theorem}

Note that when $P$ is not a square-free polynomial, the bound becomes trivial since ${\rm Disc}(P)$ vanishes. 
This situation has been managed by Eigenwillig (\cite[Theorem 3.9]{eigen})
through the use of subdiscriminants (see \cite[Section 4.2]{BPR}), obtaining a generalized
version of the Davenport-Mahler bound, as follows: 

\begin{theorem}[Generalized Davenport-Mahler bound]\label{th:GDM}
Let $P \in \C[X]$ be a polynomial of degree $d$ with exactly $r$ distinct complex roots.
Let $G=(V,E)$ be a directed graph whose vertices $\{v_1, \dots, v_k\}$ are a subset of the roots of $P$ 
such that: 
\begin{enumerate}
 \item if $(v_i, v_j) \in E$, then $|v_i| \le |v_j|$,
 \item $G$ is acyclic,
 \item the in-degree of any vertex is at most 1. 
\end{enumerate}
Then
$$
\prod_{(v_i, v_j) \in E} |v_i - v_j| \ \ge \ |{\rm sDisc}_{d-r}(P)|^{1/2} \, {\rm M}(P)^{-(r-1)} \, \Big(\frac{r}{\sqrt{3}}\Big)^{-\#E} \,r^{-r/2} \, 
\Big(\frac13\Big)^{\min \{d, 2d-2r\}/6}.
$$
\end{theorem}

It is clear that if $P$ is a square-free polynomial, then $r = d$ and
the bound by Eigenwillig is exactly the classical Davenport-Mahler bound. 

One of the main applications of the Davenport-Mahler bound in both its classical and generalized version 
is its use in algorithmic complexity estimation as for instance in \cite{DuShYap, ESY, KerSag}.

The main result in this paper is that the Generalized Davenport-Mahler bound holds for arbitrary graphs
(undirected, no loops, no multiple edges)
with 
vertices on the set of roots of $P$. 
More precisely:

\begin{theorem}\label{th:main}
Let $P \in \C[X]$ be a polynomial of degree $d$ with exactly $r$ distinct complex roots.
Let $G=(V,E)$ be a graph whose vertices $\{v_1, \dots, v_k\}$ are a subset of the roots of $P$.
Then
$$
\prod_{(v_i, v_j) \in E} |v_i - v_j| \ \ge \ 
|{\rm sDisc}_{d-r}(P)|^{1/2} \, {\rm M}(P)^{-(r-1)} \, \Big(\frac{r}{\sqrt{3}}\Big)^{-\#E} \,r^{-r/2} \, 
\Big(\frac13\Big)^{\min \{d, 2d-2r\}/6}.
$$
\end{theorem}

In order to prove Theorem \ref{th:main}, we revisit the classical proofs and the new ingredient is the use of 
divided diferences to manage the cases where the assumptions in previous formulations do not hold. 

Finally, after proving Theorem \ref{th:main}, we include some remarks and applications. 

\section{Proof of the results} 

First, we recall the definition of divided diferences. 

\begin{defn} For $f:\C \to \C$ and $v_1, \dots, v_n \in \C$ with $v_i \ne v_j$ if $1 \le i < j \le n$, the 
divided difference $f[v_1, \dots, v_n] \in \C$ is defined inductively in $n$ by
$$
f[v_1] = f(v_1)
$$
if $n=1$ and 
$$
f[v_1, \dots, v_n] =\frac{f[v_1, \dots, v_{n-1}] - f[v_2, \dots, v_n]}{v_1 - v_n}
$$
if $n > 1$. 

For $F:\C \to \C^m$ given by $F(z) = (f_1(z), \dots, f_m(z))$ and $v_1, \dots, v_n \in \C$ with $v_i \ne v_j$ if $1 \le i < j \le n$, 
the 
divided difference $F[v_1, \dots, v_n]$ is difined as
$$
F[v_1, \dots, v_n] = (f_1[v_1, \dots, v_n], \dots, f_m[v_1, \dots, v_n]) \in \C^m.
$$
\end{defn}

The only properties we will use concerning divided diferences are stated
in the next two lemmas. We omit their proofs since they can both be easily done by induction on $n$. 
We refer the reader to \cite[Chapter 6]{CheKin} for further properties of divided diferences and their 
use in polynomial interpolation. 

\begin{lemma} \label{lema_dif_div_1}
For $F:\C \to \C^m$ and $v_1, \dots, v_n \in \C$ with $v_i \ne v_j$ if $1 \le i < j \le n$, 
$F[v_1, \dots, v_n]$ is the 
linear combination of $F(v_1), \dots, F(v_n)$ given by
$$
F[v_1, \dots, v_n] = \sum_{h = 1}^{n} \Big(  \prod_{k =1 \atop k\ne h}^{n}
\
\frac{1}{v_h - v_k} \ \Big) F(v_h) .
$$ 
\end{lemma}

\begin{lemma} \label{lema_dif_div_2}
For $p \in \N_0$, $f:\C \to \C$ given by $f(z) = z^p$, and $v_1, \dots, v_n \in \C$ with $v_i \ne v_j$ if $1 \le i < j \le n$, 
$$
f[v_1, \dots, v_n] = \left \{ \begin{array}{cl}
                                         \displaystyle{
                                         \sum_{(t_1, \dots, t_n) \in \N_0^n \atop t_1 + \dots + t_n = p-n+1} \ \prod_{j = 1}^n v_j^{t_j}} & \hbox{if } n \le p+1, \cr
                                              0 & \hbox{if } n \ge  p + 2. \cr
                                        \end{array}
\right.
$$ 
\end{lemma}

We will also use the following lemma, whose proof is again omitted since it 
can be easily done by induction on $r$.

\begin{lemma}\label{aux_lemm} For $d, r \in  \N_0$ with $d \le r-1$, 
$$
\left( \sum_{i=d}^{r-1}\binom{i}{d}^2 \right)^{1/2} \ \le \ \binom{r-1}{d}\left(\frac{r+d}{2d+1}\right)^{1/2} \ \le 
\  \Big(\frac{r}{\sqrt{3}}\Big)^d r^{1/2}.
$$
\end{lemma}

Finally, before proving our main result, we recall \cite[Lemma 3.8]{eigen}.

\begin{lemma}\label{borrowed_lemma} If $m_1, \dots, m_r \in \N$ and $
\sum_{i = 1}^r m_i = d$, then 
$$\displaystyle\prod_{i = 1}^r m_i \le 3^{\min\{d, 2d-2r\}/3}.$$ 
\end{lemma}

We can now give the proof of our main result.

\begin{proof}{Proof of Theorem \ref{th:main}:}
Let $P(X) = a_d\prod_{j = 1}^r(X-v_j)^{m_j}\in \C[X]$ 
with $v_i \ne v_j$ if $1 \le i < j \le r$, $m_i \in \N$ for $1 \le i \le r$. 
It is easy to see that the result holds if $r=1$, so from now we suppose $r \ge 2$. 
Without loss of generality, we suppose also that $V = \{v_1, \dots, v_r\}$
and that the roots of $P$ are numbered in such a way that
$$
|v_1| \le \dots \le |v_r|.
$$
We give a direction to each edge in $E$:
if $e$ is an edge joining $v_i$ and $v_j$ with $i < j$, we
consider $e = (v_i, v_j)$ as the oriented edge going
from $v_i$ to $v_j$. Note that now $G = (V, E)$ satisfies conditions
\emph{1} and \emph{2} in Theorems \ref{th:DM} and \ref{th:GDM}.
We consider the edges in $E$ listed by
$$
e_1 =(v_{\alpha(1)}, v_{\beta(1)}), \dots, e_{\#E} =(v_{\alpha(\#E)}, v_{\beta(\#E)}).
$$
Finally, for $1 \le j \le r$, let $d_j \in \N_0$ be the in-degree of the vertex $v_j$.
Note that $d_1 = 0$ since there is no edge finishing in $v_1$, and $d_j \le r-1$ for $1 \le j \le r$.

As seen in \cite[Section 4.2]{BPR},
\begin{equation} \label{ineq_1}
\left|\sD_{d-r}(P)\right|^{1/2} 
= \left|a_d\right|^{r-1} \left( \prod_{j = 1}^r m_j \right)^{1/2}  \prod_{1 \le i < j \le r} |v_i - v_j|. 
\end{equation}
On the other hand,  
\begin{equation} \label{ineq_2}
\prod_{1 \le i < j \le r} |v_i - v_j| = | \det  W |
\end{equation}
where $W$ is the Vandermonde matrix
$$
W = \left( \begin{array}{cccc}
  1 & v_1 & \dots & v_1^{r-1} \cr
  1 & v_2 & \dots & v_2^{r-1} \cr
  \vdots & \vdots & & \vdots \cr
  1 & v_r & \dots & v_r^{r-1} \cr
             \end{array}
    \right) \in \C^{r \times r}.
$$
We consider $F:\C \to \C^r$, $F(z)=(1,z,\dots, z^{r-1})$ and define
a sequence of matrices $W_r, W_{r-1}, \dots, W_1$ in $\C^{r \times r}$. 
First, we define $W_r = W$. 
Then, for fixed
$j = r, \dots, 2$, 
once $W_j$ is defined, 
we only modify its $j$-th row (if any)
in order to define $W_{j-1}$, as follows:
we take the (possibly empty) sublist of edges $e_{k_1}, \dots, e_{k_{d_j}}$
finishing in $v_j$ and  
take as the $j$-th row of $W_{j-1}$ the divided difference
$$
F[v_{\alpha(k_1)}, \dots ,v_{\alpha(k_{d_j})},v_j].
$$
Note that the $j$-th row of $W_j$ equals the $j$-th row of $W$, which is
$F(v_j)$; and since for $1 \le i \le d_j$, $\alpha(k_i) < \beta(k_i) = j$, 
the $\alpha(k_i)$-th row of $W_j$ equals the $\alpha(k_i)$-th row of $W$, which is
$F(v_{\alpha(k_i)})$. Then, by Lemma \ref{lema_dif_div_1} we have that 
$$
\det W_j = \det W_{j-1} \, \prod_{i = 1}^{d_j} (v_j - v_{\alpha(k_i)}).
$$
In this way, we can prove by reverse induction in $j$ that for $j = r, \dots, 2$, 
$$
\det W = \det W_{j-1} \, \prod_{e \in E \atop \beta(e) \ge j} (v_{\beta(e)} - v_{\alpha(e)}),
$$
and at the end we obtain
\begin{equation} \label{ineq_3}
\det W =  \det W_1 \,  \prod_{e \in E} (v_{\beta(e)} - v_{\alpha(e)}).
\end{equation}

The next  step is to bound $|\det W_1|$ using Hadamard inequality. For $1 \le j \le r$, 
keeping the notation of the above paragraphs, the 
$j$-th row of $W_1$
is 
$
F[v_{\alpha(k_1)}, \dots ,v_{\alpha(k_{d_j})},v_j]
$
and 
by Lemma \ref{lema_dif_div_2} its norm equals
$$
\left(\sum_{i=d_j}^{r-1} 
\Big|
\sum_{(t_{1}, \dots, t_{{d_j}},t_{d_j+1}) \in \N_0^{d_j+1} \atop t_{1} + \dots + t_{{d_j}}+t_{d_j+1} = i-d_j} 
\left( \prod_{l = 1}^{d_j} v_{\alpha(k_{l})}^{t_{{l}}} \right) v_j^{t_{d_j+1}} 
\Big|^2 \right)^{1/2}
\leq 
$$
$$
\leq
\left(\sum_{i=d_j}^{r-1}
\binom{i}{d_j}^2 |v_j|^{2(i-d_j)} \right)^{1/2}
\leq
\left(\sum_{i=d_j}^{r-1} \binom{i}{d_j}^2\right)^{1/2} \max \{ 1, |v_j| \}^{r-1-d_j} \le 
$$
$$
\le \Big(\frac{r}{\sqrt{3}}\Big)^{d_j} r^{1/2} \max \{ 1, |v_j| \}^{r-1-d_j}
$$
by Lemma \ref{aux_lemm}. 
By Hadamard inequality, 
\begin{equation} \label{ineq_4}
| \det W_1 | \le \prod_{j = 1}^r 
 \Big(\frac{r}{\sqrt{3}}\Big)^{d_j} r^{1/2}  \max \{ 1, |v_j| \}^{r-1-d_j}  
 =  \Big(\frac{r}{\sqrt{3}}\Big)^{\#E} r^{r/2} \prod_{j = 1}^r\max \{ 1, |v_j| \}^{r-1-d_j}.
\end{equation}

Finally, using equations (\ref{ineq_1}), (\ref{ineq_2}), (\ref{ineq_3}), (\ref{ineq_4}) and 
Lemma \ref{borrowed_lemma},
$$
\prod_{(v_i, v_j) \in E} |v_i - v_j| = 
\prod_{e \in E} |v_{\beta(e)} - v_{\alpha(e)}| = |\det W| |\det(W_1)|^{-1} \ge 
$$
$$
\ge
\left|\sD_{d-r}(P)\right|^{1/2} 
\left|a_d\right|^{-(r-1)}  \left( \prod_{j = 1}^r\max \{ 1, |v_j| \}^{-(r-1-d_j)} \right)
\Big(\frac{r}{\sqrt{3}}\Big)^{-\#E}r^{-r/2} \left( \prod_{j = 1}^r m_j \right)^{-1/2}\ge
$$
$$
\ge |{\rm sDisc}_{d-r}(P)|^{1/2} \, {\rm M}(P)^{-(r-1)} \, \Big(\frac{r}{\sqrt{3}}\Big)^{-\#E} \,r^{-r/2} \, 
\Big(\frac13\Big)^{\min \{d, 2d-2r\}/6}
$$
as we wanted to prove.
\end{proof}

We include below some remarks considering cases in which the bound in Theorem \ref{th:main} can be 
slightly improved.

\begin{remark}\label{rem:degree}
Following the notation in Theorem \ref{th:main}, for $1 \le j \le r$ let $\tilde d_j$ be the total degree of 
vertex $v_j$ and let $\tilde d = \min \{\tilde d_j \, | \, 1 \le j \le r\}$. If $P$ is a monic polynomial 
then 
$$
\prod_{(v_i, v_j) \in E} |v_i - v_j| \ge |{\rm sDisc}_{d-r}(P)|^{1/2} \, {\rm M}(P)^{-(r-1-\frac12 \tilde d)} \, \Big(\frac{r}{\sqrt{3}}\Big)^{-\#E} \,r^{-r/2} \, 
\Big(\frac13\Big)^{\min \{d, 2d-2r\}/6}.
$$
Indeed, 
taking into account that $|v_{\alpha(e)}| \le |v_{\beta(e)}|$ for every $e \in E$, we
change the last part of the proof of Theorem \ref{th:main} as follows: 
$$
\prod_{(v_i, v_j) \in E} |v_i - v_j| = 
\prod_{e \in E} |v_{\beta(e)} - v_{\alpha(e)}| = |\det W| |\det(W_1)|^{-1} \ge 
$$
$$
\ge
\left|\sD_{d-r}(P)\right|^{1/2} 
\left( \prod_{j = 1}^r\max \{ 1, |v_j| \}^{-(r-1-d_j)} \right)
\left(\prod_{e \in E}\frac{\max\{1, |v_{\alpha(e)}|\}^{1/2}}{\max\{1, |v_{\beta(e)}|\}^{1/2}}\right) 
\Big(\frac{r}{\sqrt{3}}\Big)^{-\#E}r^{-r/2} \left( \prod_{j = 1}^r m_j \right)^{-1/2}\ge
$$
$$
\ge
\left|\sD_{d-r}(P)\right|^{1/2} 
\left( \prod_{j = 1}^r\max \{ 1, |v_j| \}^{-(r-1-\frac12 \tilde d_j)} \right)
\Big(\frac{r}{\sqrt{3}}\Big)^{-\#E}r^{-r/2} \Big(\frac13\Big)^{\min \{d, 2d-2r\}/6}
\ge
$$
$$
\ge |{\rm sDisc}_{d-r}(P)|^{1/2} \, {\rm M}(P)^{-(r-1-\frac12 \tilde d)} \, \Big(\frac{r}{\sqrt{3}}\Big)^{-\#E} \,r^{-r/2} \, 
\Big(\frac13\Big)^{\min \{d, 2d-2r\}/6}.
$$
\end{remark}

The next remark considers the case
where
a number of small distances is guaranteed by some extra information (possibly coming from numerical computations). 
It could be particularly useful to bound the minimial distance between different roots, taking $E$ as 
the set with only one edge joining a pair of closest roots.

\begin{remark}
Following the notation in Theorem \ref{th:main}, suppose that $r > 2$ and
that there exist at least $k$ distinct pairs of roots 
$(v_{\gamma(1)}, v_{\delta(1)}), \dots, (v_{\gamma(k)}, v_{\delta(k)})$
whose distance is less than $\frac{\sqrt{3}}{r}$
(not necesarily these pairs of roots should be connected by edges in $E$). 
For $1 \le i \le k$, let $\Delta_i  > 0$ such that 
$$
|v_{\gamma(i)} - v_{\delta(i)}| \le \Big(\frac{\sqrt{3}}{r}\Big)^{1 + \Delta_i}
$$
and renumber these pairs such that
$$
\Delta_1 \ge \dots \ge \Delta_k. 
$$
Then, if $\#E < k$, 
$$
\prod_{(v_i, v_j) \in E} |v_i - v_j| \ \ge \ |{\rm sDisc}_{d-r}(P)|^{1/2}  {\rm M}(P)^{-(r-1)} 
\Big(\frac{r}{\sqrt{3}}\Big)^{-\#E + \Delta_{\#E + 1} + \cdots + \Delta_k} r^{-r/2}  
\Big(\frac13\Big)^{\min \{d, 2d-2r\}/6}.
$$
Indeed, suppose that 
$$0 < \omega_1 \le \dots \le \omega_{\binom{r}2}$$
are the ordered distances between pairs of roots of $P$. 
By the assumptions, 
for $1 \le i \le k$ there are at least $i$ distances less than or equal to $\Big(\frac{\sqrt{3}}{r}\Big)^{1 + \Delta_i}$
and then we have that 
 $\omega_i \le \Big(\frac{\sqrt{3}}{r}\Big)^{1 + \Delta_i}$.
Consider $\tilde E$ the set of $k$ edges whose lengths are $\omega_1, \dots, \omega_k$.
Then, applying the bound in Theorem \ref{th:main} to $G = (V, \tilde E$) we obtain
$$
\prod_{(v_i, v_j) \in E} |v_i - v_j| \ge \prod_{i = 1}^{\#E} \omega_i
= \Big(\prod_{(v_i, v_j) \in \tilde E} |v_i - v_j|\Big) \Big(\prod_{i = \#E + 1}^k \omega_i^{-1}\Big) 
\ge
$$
$$
\ge
|{\rm sDisc}_{d-r}(P)|^{1/2} \, {\rm M}(P)^{-(r-1)} \, \Big(\frac{r}{\sqrt{3}}\Big)^{-k} r^{-r/2}  
\Big(\frac13\Big)^{\min \{d, 2d-2r\}/6} \Big(\prod_{i = \#E + 1}^k \Big(\frac{r}{\sqrt{3}}\Big)^{1 + \Delta_i} \Big) =
$$
$$
=
 |{\rm sDisc}_{d-r}(P)|^{1/2}  {\rm M}(P)^{-(r-1)} 
\Big(\frac{r}{\sqrt{3}}\Big)^{-\#E + \Delta_{\#E + 1} + \cdots + \Delta_k} r^{-r/2}  
\Big(\frac13\Big)^{\min \{d, 2d-2r\}/6}.
$$
\end{remark}

Finally, as an application of Theorem \ref{th:main}, we give a simplified proof of \cite[Theorem 9]{KerSag}
with smaller constants. 

\begin{theorem}\label{th:appl}
Let $P \in \C[X]$ be a polynomial of degree $d$ with exactly $r \ge 2$ distinct complex roots
and let $V = \{v_1, \dots, v_r\} \subset \C$ be the set of roots.
For any root $v$ of $P$, we denote by ${\rm sep}(P, v)$ the distance from $v$ to (one of) its closest different root of $P$.
Then, for any $V' \subset V$, 
$$
\prod_{v \in V'}{\rm sep}(P, v) \ge
|{\rm sDisc}_{d-r}(P)| \, {\rm M}(P)^{-2(r-1)}  \Big(\frac{r}{\sqrt{3}}\Big)^{-\#V'} \,r^{-r} \, 
\Big(\frac13\Big)^{\min \{d, 2d-2r\}/3}.
$$
\end{theorem}

\begin{proof}{Proof:}
For each $v \in V$, we take $\tilde v$ as (one of) its closest different root of $P$.
We consider the multigraph $G=(V,E)$ where $E$ is the multiset of 
edges of type $(v, \tilde v)$ with $v \in V'$. Note that each edge in $E$ can occur at most $2$ times (one for each of its vertex). 
We divide $E$ in two sets $E_0$ and $E_1$, with $E_0$ having all the elements in $E$ and $E_1$ having the elements that ocur twice in $E$. 
Applying Theorem \ref{th:main} to $(V, E_0)$ and $(V, E_1)$ and taking into account that $\#E_0 + \#E_1 = \#V'$, we obtain 
$$
\prod_{v \in V'}{\rm sep}(P, v) =  \Big( \prod_{(v_i, v_j) \in E_0}|v_i - v_j|\Big)
\Big( \prod_{(v_i, v_j) \in E_1}|v_i - v_j|\Big)
\ge
$$
$$
\ge
|{\rm sDisc}_{d-r}(P)| {\rm M}(P)^{-2(r-1)}  \Big(\frac{r}{\sqrt{3}}\Big)^{-\#V'} \,r^{-r} \, 
\Big(\frac13\Big)^{\min \{d, 2d-2r\}/3}
$$
as we wanted to prove.
\end{proof}


\begin{thebibliography}{99}

\bibitem{BPR}  Basu, Pollack and Roy,  Algorithms in real algebraic geometry. Second edition. 
Algorithms and Computation in Mathematics, 10. \emph{Springer-Verlag, Berlin}, 2006.


\bibitem{Davenport} Davenport, Cylindrical algebraic decomposition. Technical Report 88-10, 
University of Bath, England, 1988.

\bibitem{DuShYap}  Du, Sharma and Yap, 
Amortized bound for root isolation via Sturm sequences. 
\emph{Symbolic-numeric computation}, 113--129, Trends Math., Birkh\"auser, Basel, 2007.



\bibitem{eigen}Eigenwillig, Real Root Isolation for Exact and Approximate Polynomials Using Descartes'\ Rule of Signs, 
Doctoral dissertation, Universit\"at des Saarlandes, 2008


\bibitem{ESY} Eigenwillig, Sharma and Yap,
Almost tight recursion tree bounds for the Descartes method. ISSAC 2006, 71--78, ACM, New York, 2006. 

\bibitem{John}  Johnson, Algorithms for polynomial real root isolation. 
\emph{Quantifier elimination and cylindrical algebraic decomposition} (Linz, 1993), 269--299, 
Texts Monogr. Symbol. Comput., Springer, Vienna, 1998.

\bibitem{KerSag} Kerber and Sagraloff, 
A worst-case bound for topology computation of algebraic curves. 
\emph{J. Symbolic Comput.} 47 (2012), no. 3, 239--258. 

\bibitem{CheKin}  Kincaid and Cheney, 
Numerical analysis. Mathematics of scientific computing. Second edition. 
\emph{Brooks/Cole Publishing Co., Pacific Grove, CA}, 1996.


\bibitem{Mahler} 
Mahler, An inequality for the discriminant of a polynomial.
\emph{Michigan Math. J.} 11 1964 257--262. 

\bibitem{MigSte}
Mignotte and \c Stef\u{a}nescu, 
Polynomials. 
An algorithmic approach. 
\emph{Springer Series in Discrete Mathematics and Theoretical Computer Science. Springer-Verlag Singapore, Singapore; Centre for Discrete Mathematics $\&$ Theoretical Computer Science, Auckland,} 1999.

\bibitem{Yap}  Yap, Fundamental problems of algorithmic algebra. Oxford University Press, New York, 2000. 

\end{thebibliography}
\end{document}